\documentclass[11pt,a4paper]{article}
\usepackage[utf8]{inputenc}
\usepackage{epsfig,graphicx,graphics}
\usepackage{amsmath}
\usepackage{amssymb}

\usepackage{amsthm}
\usepackage{amsmath}
\usepackage{amsfonts}
\usepackage{geometry}
\usepackage{epsfig}

\usepackage{color}
\usepackage{amssymb}
\usepackage{graphicx}
\usepackage{amsmath,amsfonts,amsthm,amstext}
\usepackage{mathrsfs}
\usepackage{hyperref}

\usepackage{tikz}

\newtheorem{teo}{Theorem}[section]
\newtheorem{lema}{Lemma}[section]
\newtheorem{cor}{Corollary}[section]
\newtheorem{propo}{Proposition}[section]
\theoremstyle{definition}
\newtheorem{defi}{Definition}[section]

\newtheorem{rek}{Remark}[section]

\newcommand{\supp}{\operatorname{supp}}

\newcommand{\inff}{\operatorname{inf}}

\newcommand{\essinf}{\operatorname{ess\,inf}}

\newcommand{\esssup}{\operatorname{ess\,sup}}

\newcommand{\diam}{\operatorname{diam}}

\newcommand{\ve}{\varepsilon}
\newcommand{\N}{\mathbb{N}}
\newcommand{\R}{\mathbb{R}}
\newcommand{\Z}{\mathbb{Z}}
\newcommand{\M}{\mathcal{M}}

\newcommand{\f}{\varphi}

\newenvironment{proof3}[1][\textit{\textbf{Proof \em(Theorem~\ref{teocentral3})}}]{\textit{#1.} }{\hfill $\Box$}

\newenvironment{proof4}[1][\textit{\textbf{Proof \em(Proposition~\ref{BGT})}}]{\textit{#1.} }{\hfill $\Box$}

\begin{document}

\bigskip

\title {Generic properties of invariant measures of full-shift systems over perfect Polish metric spaces}
\date{}
\author{Silas L. Carvalho$^{*}$  ~~and~ Alexander Condori$^{\dag}$}
\maketitle

{ \small \noindent $^{*}\,$Departamento de Matemática, Universidade Federal de Minas Gerais, Av. Ant\^onio Carlos, 6627, Belo Horizonte, Minas Gerais, PO Box 702, ZIP 31270-901, Brazil. \\ {\em e-mail:
silas@mat.ufmg.br 

\em
{\small \noindent $^{\dag\,}$Instituto de Matem\'atica y Ciencias Afines, Universidad Nacional de Ingeniería, Calle Los Bi\'ologos 245,
Lima 15012, Per\'u. \\ {\em e-mail:
acondori@imca.edu.pe 
\em
\normalsize
\maketitle
\begin{abstract}
\noindent{
	In this work, we are interested in characterizing typical (generic) dimensional properties of invariant measures associated with the full-shift system, $T$, in a product space whose alphabet is a perfect Polish metric space (thus, uncountable). More specifically, we show that the set of invariant measures with upper  Hausdorff dimension equal to zero and lower packing dimension equal to infinity is a dense $G_\delta$ subset of $\mathcal{M}(T)$, the space of $T$-invariant measures endowed with the weak topology. We also show that the set of invariant measures with upper rate of recurrence equal to infinity and lower rate of recurrence equal to zero is a $G_\delta$ subset of $\mathcal{M}(T)$. Furthermore, we show that the set of invariant measures with upper quantitative waiting time indicator equal to infinity and lower quantitative waiting time indicator equal to zero is residual in $\mathcal{M}(T)$.
}
\noindent
\end{abstract}

\noindent {AMSC: 37A05, 37B10, 28A78}

\section{Introduction}

Let $(M,\rho)$ be a complete separable (Polish) metric space, and let $S$ be its $\sigma$-algebra of Borel sets. Now, define $(X,\mathcal{B})$ as the bilateral product of a countable number of copies of $(M,S)$. 
 Note that $\mathcal{B}$ coincides with the $\sigma$-algebra of the Borel sets in the product topology. Let $d$ be any metric in $X$ which is compatible with the product topology (that is, $d$ induces an equivalent topology). It is straightforward to show that $(X,d)$ is also a Polish metric space. 

 One can define in $X$ the so-called full-shift operator, $T$, by the action
 \[Tx=y,\]
 where $x=(\ldots,x_{-n},\ldots,x_n,\ldots)$, $y=(\ldots,y_{-n},\ldots,y_n,\ldots)$, and for each $i\in\mathbb{Z}$, $y_i=x_{i-1}$. $T$ is clearly an homeomorphism of $X$ onto itself. We choose $d$ in such a way that $T$ and $T^{-1}$ are Lipschitz transformations; set, for instance, for each $x,y\in X$, 
\begin{equation}\label{metric}
  d(x,y) =   \sum_{|n|\geq 0} \frac{1}{2^{|n|}}  \frac{\rho(x_n,y_n)}{1+\rho(x_n,y_n)}.
\end{equation}

Let $\mathcal{M}(T)$ be the space of all $T$-invariant probability measures, endowed with the weak topology (that is the coarsest topology for which the net $\{\mu_\alpha\}$ converges to $\mu$ if, and only if, for each bounded and continuous function $f$, $\int fd\mu_\alpha\rightarrow \int fd\mu$). Since $X$ is Polish, $\mathcal{M}(T)$ is also a Polish metrizable space (see~\cite{Sigmundlibro}).

Given $\mu\in\M(T)$, the triple $(X,T,\mu)$ is called an $M$-valued discrete stationary stochastic process (see~\cite{Parthasarathy1961,Sigmund1971,Sigmund1974}; see also~\cite{FuLu} for a discussion of the role of such systems in the study of continuous self-maps over general metric spaces).

The study of generic properties (in Baire's sense; see Definition~\ref{generic}) of such $M$-valued discrete stationary stochastic processes goes back to the works of Parthasarathy~\cite{Parthasarathy1961} (regarding ergodicity) and Sigmund~\cite{Sigmund1971,Sigmund1974} (regarding positivity of the measure on open sets, zero entropy of the measure for $M=\mathbb{R}$).

In the last decade, some studies about the full-shift system over an uncountable alphabet have been performed; more specifically, we can mention the works about the Gibbs state in Ergodic Optimization (a generalization of several results of the classical theory of thermodynamic formalism). In \cite{Lopes2011}, given an observable $A: M^{\Z} \rightarrow \mathbb{R}$, where $M$ is a connected and compact manifold, and given the temperature $T$, the authors have studied the main properties of the Gibbs state $\hat{\mu}_{\frac{1}{T}A}$, considering the Ruelle operator associated with $\frac{1}{T}A$. They also have analyzed selection problems when $T\rightarrow 0$.  In~\cite{Lopes2015}, the authors have studied  the shift acting on $M^{\N}$ 
and considered a general a priori probability for defining the Ruelle operator. They have studied the pressure problem for a H\"older potential $A$ and its relation with eigenfunctions and eigenprobabilities of the Ruelle operator. They also have analysed the case in which $T\rightarrow 0$ and have shown some selection results. 


In this work, we are interested in the generic dimensional properties of such $T$-invariant measures, more specifically, in their Hausdorff and packing dimensions (generally called fractal dimensions). 
Hence, we recall some basic definitions involving Hausdorff and packing measures, giving a special treatment to the packing dimension of a measure defined in a general metric space (open and closed balls in $\R^N$ possess nice regularity properties; for example, the diameter of a ball is twice its radius, and open and closed balls of the same radius have the same diameter. In arbitrary metric spaces, the possible absence of such regularity properties means that the usual measure construction based on diameters can lead to packing measures with undesirable features, as was observed by Cutler in \cite{Cutler1995}). 

For a dynamical system, the fractal dimensions of an invariant measure provide more relevant information than the fractal dimensions of its invariant sets, or even the fractal dimensions of its topological support; the point is that invariant sets and topological supports usually contain superfluous sets (in the sense that they have zero measure). Thus, by establishing the fractal dimensions of invariant measures, one has a more precise information about the structure of the relevant sets (of positive measure) of a dynamical system (see~\cite{Barbaroux,PesinT1995,Pesin1997} for a more detailed discussion).

In what follows, $(X,d)$ is an arbitrary metric space and $\mathcal{B}=\mathcal{B}(X)$ is its Borel $\sigma$-algebra.

\begin{defi}[radius packing $\phi$-premeasure, \cite{Cutler1995}] Let $\emptyset\neq E\subset X$, and let $0 < \delta < 1$. A $\delta\textrm{-}$\emph{packing} of $E$ is a countable collection of disjoint closed balls $\{B(x_k,r_k)\}_k$ with centers $x_k\in E$ and radii satisfying  $0<r_k\leq \delta/2$, for each $k\in\N$ (the centers $x_k$ and radii $r_k$ are considered part of the definition of the $\delta$-packing). Given a measurable function $\phi$, the radius packing ($\phi,\delta$)\textrm{-}premeasure of $E$ is given by the law
\begin{eqnarray*}
P^{\phi}_{\delta}(E)=\sup\left\{\sum _{k=1}^{\infty} \phi(2r_k)\mid \{B(x_k,r_k)\}_k \mbox{ is a } \delta\textrm{-}\mbox{packing} \mbox{ of E}\right\}.
\end{eqnarray*}
Letting $\delta\to 0$, one gets the so-called \emph{radius packing} $\phi\textrm{-}$\emph{premeasure}
\begin{eqnarray*}
P^{\phi}_{0}(E)=\lim_{\delta \to 0} P^{\phi}_{\delta}(E).
\end{eqnarray*}
One sets $P^{\phi}_{\delta}(\emptyset)=P^{\phi}_{0}(\emptyset)=0$.
\end{defi}
It is easy to see that $P^{\phi}_{0}$ is non-negative and monotone. Moreover, $P^{\phi}_{0}$ generally fails to be countably sub-additive. One can,
however, build an outer measure from $P^{\phi}_{0}$ by applying Munroe's Method I construction, described  both in \cite{Munroe} and  \cite{Rogers}. This leads to the following definition.

\begin{defi}[radius packing $\phi\textrm{-}$measure, \cite{Cutler1995}]
The radius packing $\phi\textrm{-}$measure of $E\subset X$ is defined to be 
\begin{eqnarray}
\label{pakmeasure}
P^{\phi}(E)=\inf\left\{ \sum_{k}P^{\phi}_{0}(E_k)\mid E\subset \bigcup_k E_k \right\}.
\end{eqnarray}
The infimum in (\ref{pakmeasure}) is taken over all countable coverings $\{E_k\}_k$ of $E$. It follows that $P$ is an outer measure on the subsets of $X$.
\end{defi}
In an analogous fashion, one may define the Hausdorff $\phi\textrm{-}$measure. The theory of Hausdorff measures in general metric spaces is a well\textrm{-}explored topic; see, for example, the treatise by Rogers~\cite{Rogers}.
\begin{defi}[Hausdorff $\phi\textrm{-}$measure, \cite{Cutler1995}]
For $E\subset X$, the outer measure $H^{\phi}(E)$ is defined by
\begin{eqnarray}
\label{hausmeasure}
H^{\phi}(E)=\lim_{\delta\to 0}\inf \left\{\sum_{k=1}^{\infty} \phi(\diam(E_k))\mid \{E_k\}_k \mbox{ is a } \delta\textrm{-}\mbox{covering} \mbox{ of } E \right\},
\end{eqnarray}
where a $\delta\textrm{-}$\emph{covering} of E  is any countable collection $\{E_k\}_k$ of subsets of $X$ such that, for each $k\in\N$, $E\subset \bigcap_k E_k$ and $\diam(E_k)\leq \delta$. If no such $\delta\textrm{-}$covering exists, one sets $H_{\phi}(E)=\inf \emptyset=\infty$.
\end{defi}

Of special interest is the situation where given $\alpha>0$, one sets $\phi(t)=t^{\alpha}$. In this case, one uses the notation $P^{\alpha}_0$, and refers to $P^{\alpha}_0(E)$ as the $\alpha\textrm{-}$packing premeasure of $E$. Similarly, one uses the notation $P^{\alpha}(E)$ for the packing $\alpha\textrm{-}$measure of $E$, and  $H^{\alpha}(E)$ for the $\alpha\textrm{-}$ Hausdorff measure of $E$.

\begin{defi}[Hausdorff and packing dimensions of a set, \cite{Cutler1995}]
Let $E\subset X$. One defines the Hausdorff dimension of $E$  to be the critical point
\begin{eqnarray*}
\dim_H(E)=\inf\{\alpha>0\mid h^{\alpha}(E)=0\};
\end{eqnarray*}
one defines the packing dimension of $E$ in the same fashion.
\end{defi}

We note that $\dim_H(X)$ or $\dim_P(X)$ may be infinite for some metric space $X$. One can show  that, for each $E\subset X$, $\dim_H(E) \leq \dim_P(E)$ (see Theorem 3.11(h) in \cite{Cutler1995}), and this inequality is in general strict.

\begin{defi}[lower and upper packing and Hausdorff dimensions of a measure,\cite{Mattila}]\label{HPdim}
  Let $\mu$ be a positive Borel measure on $(X,\mathcal{B})$. 
  The lower and upper packing and Hausdorff dimension of $\mu$ are defined, respectively, by 
\begin{eqnarray*}
\dim_{K}^-(\mu)&=&\inf\{\dim_{K}(E)\mid \mu( E)>0, ~ E\in \mathcal{B}\},\\
\dim_{K}^+(\mu)&=&\inf\{\dim_{K}(E)\mid \mu(X\setminus E)=0, ~ E\in \mathcal{B}\},
\end{eqnarray*}
where $K$ stands for $H$ (Hausdorff) or $P$ (packing). If $\dim_{K}^-(\mu)=\dim_{K}^+(\mu)$, one denotes the common value by $\dim_{K}(\mu)$. 
\end{defi}

Let $\mu$ be a positive finite Borel measure on $X$. One defines the upper and lower local dimensions of $\mu$ at $x\in X$ by
$$\overline{d}_{\mu}(x)=\limsup_{\ve\to 0}\frac{\log \mu(B(x,\ve))}{\log \ve} ~~\mbox{ and }~~ \underline{d}_{\mu}(x)=\liminf_{\ve\to 0}\frac{\log \mu(B(x,\ve))}{\log \ve},$$ 
if, for every $\ve> 0$, $\mu(B(x;\ve))>0$; if not, $\overline{\underline{d}}_\mu(x):=+\infty$.

The next result shows that the essential infimum of the lower (upper) local dimension of a probability measure equals its lower Hausdorff (packing) dimension, whereas the essential supremum of its lower (upper) local dimension equals its upper Hausdorff (packing) dimension; see Appendix for its  proof. 

\begin{propo} 
\label{BGT} 
Let $\mu$ be a probability measure on $X$. Then,
\begin{eqnarray*}
\mu\textrm{-}\essinf \underline{d}_{\mu}(x)=\dim_H^-(\mu)\leq \mu\textrm{-}\esssup \underline{d}_{\mu}(x)= \dim_H^+(\mu), \\
 \mu\textrm{-}\essinf \overline{d}_{\mu}(x)=\dim_P^-(\mu)\leq \mu\textrm{-}\esssup \overline{d}_{\mu}(x)= \dim_P^+(\mu). 
\end{eqnarray*}
\end{propo}

We are also interested in the polynomial returning rates of the $T$-orbit of a given point to arbitrarily small neighborhoods of itself (this gives, in some sense, a quantitative description of  Poincar\'e's recurrence). This question was posed and studied by Barreira and Saussol in \cite{Barreiral2001} (see also~\cite{Hongwei,Barreira2002,Hu-Xue-Xu} for further motivations). Given a separable metric space $X$ and a Borel measurable transformation $T$, they have defined the lower and upper recurrence rates of $x\in X$ in the following way: for each fixed $r>0$, let
\begin{eqnarray*}
\tau_r(x)=\inf\{k\in\N\mid T^k x \in \overline{B}(x,r)\}
\end{eqnarray*}
be the  return time of a point $x \in X$ into the closed ball $\overline{B}(x, r)${\footnote{Actually, the definition presented in~\cite{Barreiral2001} uses open balls; it is straightforward to show that they coincide.}}; then, \begin{eqnarray*}
\underline{R}(x)=\liminf_{r\to 0}\frac{\log \tau_r(x)}{-\log r} ~~\mbox{ and } ~~  \overline{R}(x)=\limsup_{r\to 0}\frac{\log \tau_r(x)}{-\log r}
\end{eqnarray*}
are, respectively, the lower and upper recurrence rates of $x \in X$. Note that $\tau_r(x)$ may be infinite on a set of zero $\mu$-measure.

Barreira and Saussol have showed (Theorem 2 in \cite{Barreiral2001}) that $\underline{R}(x)\leq \dim_H^+(\mu)$ for $\mu$-a.e.$\;x\in X$. They also have showed, when $X\subset\R^n$,  
that $\underline{R}(x)\leq \underline{d}_{\mu}(x)$, and that $\overline{R}(x)\leq \overline{d}_{\mu}(x)$ for $\mu$-a.e.$\;x\in X$.  Later, Saussol has showed in \cite{Saussol} that, under the hypotheses that $T$ is a Lipschitz transformation, $h_{\mu}(T)>0$, and that the decay of the correlations of $(X,T,\mu)$ is super-polynomial, $\underline{R}(x)= \underline{d}_{\mu}(x)$, and $\overline{R}(x)= \overline{d}_{\mu}(x)$ for $\mu$-a.e.$\;x\in X$. In fact, it is known (see~\cite{Galatolo2011} for the precise statement) that if the decay of the correlations of $(X,T,\mu)$ is super-polynomial with respect to Lipschitz observables and $d_\mu(x):=\underline{d}_{\mu}(x)=\overline{d}_{\mu}(x)$ for $\mu$-a.e.$\;x\in X$, then $\overline{R}(x)= \underline{R}(x)=d_{\mu}(x)$ for $\mu$-a.e.$\;x\in X$. Indeed, for several rapidly mixing (``chaotic'') systems, the quantitative waiting time indicators are $\mu$-a.e.~equal to the local dimension of $\mu$ (see~\cite{Galatolo2011} for details and more references).

In this work, we present some results, for $M$-valued discrete stationary stochastic processes, relating $\underline{R}$ (respectively, $\overline{R}$) and $\underline{d}_{\mu}$ (respectively, $\overline{d}_{\mu}$).

Another dynamical aspect of $M$-valued discrete stationary stochastic processes that is explored in this work refers to the quantitative waiting time indicators, defined by Galatolo in~\cite{Galatolo} as follows: let $x,y\in X$ and let $r>0$. The first entrance time of $\mathcal{O}(x):=\{T^ix\mid i\in\mathbb{Z}\}$, the $T$-orbit of $x$, into the closed ball $\overline{B}(y,r)$ is given by
\begin{eqnarray*}
\tau_r(x,y)=\inf\{n\in \N \mid T^n(x)\in \overline{B}(y,r) \}
\end{eqnarray*} 
(note that $\tau_r(x,y)$ may be infinite on a set of zero $\mu\times\mu$-measure).

Naturally, $\tau_r(x,x)$ is just the first return time into the closed ball $\overline{B}(x,r)$. The so-called quantitative waiting time indicators are defined as
\begin{eqnarray*}
\underline{R}(x,y)=\liminf_{r\to 0}\frac{\log \tau_r(x,y)}{-\log r} ~~\mbox{ and } ~~  
\overline{R}(x,y)=\limsup_{r\to 0}\frac{\log \tau_r(x,y)}{-\log r}.
\end{eqnarray*}

Let $(X, T)$ be a dynamical system such that $X$ is a separable metric space and  $T : X \rightarrow X$ is a measurable map, and suppose that there exists a $T$-invariant measure $\mu$. Then, Theorem~4 in~\cite{Galatolo} states that, for each fixed $y\in X$,  one has
\begin{equation}\label{Gala} \underline{R}(x,y)\geq \underline{d}_{\mu}(y)\qquad
  \textrm{and}\qquad \overline{R}(x,y)\geq \overline{d}_{\mu}(y)\qquad \textrm{for}\;\;\mu\textrm{-a.e.}\,\;x\in X.
\end{equation}

Furthermore, even if $\mu$ is only a probability measure on $X$, Theorem~10 in~\cite{Galatolo} states that for each $x\in X$, one has $\underline{R}(x,y)\geq \underline{d}_{\mu}(y)$ and $\overline{R}(x,y)\geq \overline{d}_{\mu}(y)$ for $\mu$-a.e.$\;y\in X$.

Before we present our main results, some preparation is required. Recall that a subset $\mathcal{R}$ of a topological space $X$ is residual if it contains the intersection 
of a countable family, $\{U_k\}$, of open and dense subsets of $X$. A topological space $X$ is a Baire space if every residual subset of $X$ is dense in $X$. By Baire's Category Theorem, every complete metric space is a Baire space.

\begin{defi}\label{generic}
A property $\mathbb{P}$ is said to be generic in the space X if there
exists a residual subset $\mathcal{R}$ of $X$ such that each $x\in \mathcal{R}$ satisfies property $\mathbb{P}$.
\end{defi}

Note that, given a countable family of generic properties $\mathbb{P}_1,\mathbb{P}_2, \cdots$, all of them 
are simultaneously generic in $X$. This is because the family of residual sets is closed under countable intersections.

We shall prove the following results.

\begin{teo}
\label{teocentral3}
Let $(X,T,\mathcal{B})$ be the full-shift dynamical system over $X=\prod_{-\infty}^{+\infty}M$, where the alphabet $M$ is a perfect Polish metric space. 
\begin{enumerate}
\item[\emph{I.}] The set of ergodic measures, $\M_e$, is residual in $\M(T)$.
\item[\emph{II.}] The set of invariant measures with full support, $C_X$, is a dense $G_{\delta}$ subset of $\mathcal{M}(T)$.
\item[\emph{III.}] The set $HD=\{\mu\in \M(T) \mid \dim_{H}(\mu)=0\}$ is a dense $G_{\delta}$ subset of $\mathcal{M}(T)$.
\item[\emph{IV.}] The set $PD=\{ \mu\in \M(T) \mid \dim_{P}(\mu)= +\infty\}$ is a dense $G_{\delta}$ subset of $\mathcal{M}(T)$.
\item[\emph{V.}] The set $\underline{\mathcal{R}}=\{ \mu\in \M(T) \mid \underline{R}(x)=0,$ for $\mu\textrm{-}a.e.\,x\}$ is a dense $G_{\delta}$ subset of $\mathcal{M}(T)$.
\item[\emph{VI.}] The set $\overline{\mathcal{R}}=\{ \mu\in \M(T) \mid \overline{R}(x)= +\infty,$ for $\mu\textrm{-}a.e.\,x\}$ is a dense $G_{\delta}$ subset of $\mathcal{M}(T)$.
\item[\emph{VII.}] The set $\underline{\mathscr{R}}=\{ \mu\in \M(T) \mid \underline{R}(x,y)=0,$ for $(\mu\times\mu)\textrm{-}a.e.\,(x,y)\in X\times X\}$ is a dense $G_{\delta}$ subset of $\mathcal{M}(T)$.
\item[\emph{VIII.}] The set $\overline{\mathscr{R}}=\{ \mu\in \M(T) \mid \overline{R}(x,y)= +\infty,$ for $(\mu\times\mu)\textrm{-}a.e.\,(x,y)\in X\times X\}$ is residual in $\mathcal{M}(T)$.
\end{enumerate}
\end{teo}

Item I was proved by Oxtoby in \cite{Oxtoby1963} and Parthasarathy in \cite{Parthasarathy1961}, using the fact that the set of $T$-periodic or $T$-closed orbit measures (that is, measures of the form $\frac{1}{k_x}\sum_{i=0}^{k_x-1}\delta_{T^ix}(\cdot)$, where $x$ is a $T$-periodic point of period $k_x$) is dense in $\M(T)$. Sigmund has proved  item II in \cite{Sigmund1974} (see also~\cite{Sigmundlibro}). We have opted to include these results in Theorem~\ref{teocentral3} since they are used in the proofs of items III-VIII, which comprise our main contributions to the problem. The proofs of items III and IV are presented in Section~\ref{TVI-VII}, whereas the proofs of items V-VIII are presented in Section~\ref{TVIII-IX}.

As a direct consequence of Theorem \ref{teocentral3}, we have obtained for typical ergodic measures, that is, for $\mu\in \mathcal{RD}=\underline{\mathcal{R}}\cap \overline{\mathcal{R}}\cap PD \cap HD$, some relations between $\underline{R}$ (respectively, $\overline{R}$) and $\underline{d}_{\mu}$ (respectively, $\overline{d}_{\mu}$), similar to those obtained by Saussol and Barreira (as discussed above).
\begin{cor}
Let $M$ be a perfect Polish metric space, and let $\mu\in \mathcal{RD}\subset \M_e$. Then, $\underline{R}(x)=\underline{d}_{\mu}(x)=0$  and $\overline{R}(x)=\overline{d}_{\mu}(x)=\infty$, for $\mu$-a.e.$\;x\in X$.
\end{cor}

It follows from items III and IV in Theorem \ref{teocentral3} 
that there exists a dense $G_{\delta}$ set, $\mathcal{D}:= PD\cap HD \subset \M_e$, such that each $\mu\in\mathcal{D}$ is somewhat similar, in one hand, to a ``uniformly distributed'' measure, whose lower packing dimension is maximal (for instance, when $X=[0,1]^{\Z}$, the shift Bernoulli measure $\Lambda=\prod_{-\infty}^{+\infty}\lambda$, where $\lambda$ is the Lebesgue measure on $[0,1]$, is an uniformly distributed measure, whose lower packing dimension is infinite) and in the other hand, to a purely point measure, whose upper Hausdorff dimension is zero. 

Moreover, by Definition~\ref{HPdim}, each $\mu\in\mathcal{D}$ is supported on a Borel set $Z=Z(\mu)$ such that $\dim_{top}(Z)\le\dim_{H}(Z)=0<\dim_P(Z)=\infty$, where $\dim_{top}(Z)$ stands for the topological dimension of $Z$ (see, \cite{Hurewicz} Sect. 4, page 107, for a proof of the inequality $\dim_{top}(Z)\le\dim_{H}(Z)$). Since $\dim_{top}(Z)=0$, if one also assumes that $\supp(\mu)=X$ (just take $\mu\in C_X\cap\mathcal{D}$), one concludes that $Z$ is a dense and totally disconnected set in $X$ with zero Hausdorff and infinite packing dimensions. Furthermore, one may take $Z$ as a dense $G_{\delta}$ subset of $X$ (see Proposition~\ref{localgenericpoints}).

\begin{rek}
  It is worth noting that although the packing dimension of $X$ is infinite (since, for each $\mu\in\mathcal{D}$, $\dim_P(Z)=\infty$), its topological dimension may be finite. This is not unexpected, altogether: there are examples of topological spaces where $\dim_P(X)>\dim_{top}(X)$ (see example 2.12 in~\cite{PesinT1995}, where $0=\dim_{top}(X)<\dim_P(X)$).
  \end{rek}

Items V and VI in Theorem \ref{teocentral3} 
say that there exists a dense $G_{\delta}$ set, $\mathcal{R}:=\underline{\mathcal{R}}\cap \overline{\mathcal{R}}\subset \M_e$, of ergodic measures such that if $\mu\in \mathcal{R}$, then there exists a Borel set $Z$, with $\mu(Z)=1$, so that if $x\in Z$, then $\underline{R}(x)=0$ and $\overline{R}(x)=\infty$. This means that given a very large $\alpha$ and a very small $\beta$, for each $x\in Z$, one has $\underline{R}(x)\leq \beta$ and  $\overline{R}(x)\geq \alpha$. So, there exist sequences $(\ve_{k})$ and $(\sigma_{l})$ converging to zero such that, for each $k,l\in\N$, $\tau_{\ve_{k}}(x)\leq \ve_{k}^{-\beta}$ and $\tau_{\sigma_{l}}(x)\geq \sigma_{l}^{-\alpha}$, respectively. Setting, for each $k,l\in\N$, $s_k=1/\ve_k$ and $t_l=1/\sigma_l$, one has $\tau_{1/s_k}(x)\leq s_k^{\beta}$ and $\tau_{1/t_l}(x)\geq t_l^{\alpha}$, respectively.

Therefore, given $x\in Z$, there exists a time sequence (time scale) for which the first incidence of $\mathcal{O}(x)$ to one of its spherical neighborhoods (which depend on time) occurs as fast as possible (that is, it is of order 1; this means that the first return time to those neighborhoods increases sub-polynomially fast); accordingly, there exists a time sequence for which the first incidence of $\mathcal{O}(x)$ to one of its spherical neighborhoods increases as fast as possible (that is, super-polynomially fast).

The following scheme tries to depict how subsequent elements of both sequences are related. Between two consecutive elements of $(\sigma_k)$, there are several elements of $(\ve_l)$:

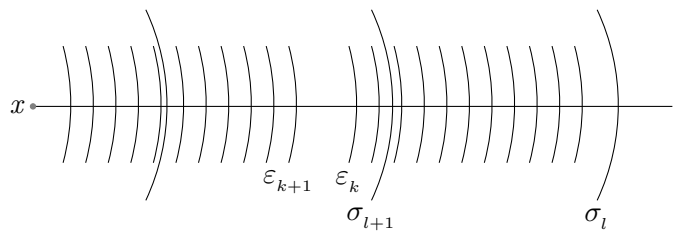
\begin{center}
\begin{tikzpicture}
\draw (-4.2,0)node {\emph{x}} (-4,0)-- (4.5,0);  
\filldraw [gray] (-4,0) circle (1pt);

\draw (-0.6,-1) node {$\ve_{_{k+1}}$} (-0.6,-0.75) arc (-15:15:3cm);
\draw (-0.9,-0.75) arc (-15:15:3cm);
\draw (-1.2,-0.75) arc (-15:15:3cm);
\draw (-1.5,-0.75) arc (-15:15:3cm);
\draw (-1.8,-0.75) arc (-15:15:3cm);
\draw (-2.1,-0.75) arc (-15:15:3cm);
\draw (-2.4,-0.75) arc (-15:15:3cm);
\draw (-2.7,-0.75) arc (-15:15:3cm);
\draw (-3,-0.75) arc (-15:15:3cm);
\draw (-3.3,-0.75) arc (-15:15:3cm);
\draw (-3.6,-0.75) arc (-15:15:3cm);

\draw (0.2,-1)node {$\ve_{_{k}}$} (0.2,-0.75) arc (-15:15:3cm);
\draw (0.5,-0.75) arc (-15:15:3cm);
\draw (0.8,-0.75) arc (-15:15:3cm);
\draw (1.1,-0.75) arc (-15:15:3cm);
\draw (1.4,-0.75) arc (-15:15:3cm);
\draw (1.7,-0.75) arc (-15:15:3cm);
\draw (2,-0.75) arc (-15:15:3cm);
\draw (2.3,-0.75) arc (-15:15:3cm);
\draw (2.6,-0.75) arc (-15:15:3cm);
\draw (2.9,-0.75) arc (-15:15:3cm);
\draw (3.2,-0.75) arc (-15:15:3cm);


\draw (-2.5,-1.25) arc (-25:25:3cm);

\draw (0.5,-1.5)node {$\sigma_{_{l+1}}$} (0.5,-1.25) arc (-25:25:3cm);

\draw (3.5,-1.5)node {$\sigma_{_{l}}$} (3.5,-1.25) arc (-25:25:3cm);
\end{tikzpicture}

\end{center}

Here, we also show that the typical measures described in Theorem \ref{teocentral3} are supported on the dense $G_{\delta}$ set $\mathfrak{R}=\{x\in X\mid \underline{R}(x)=0 \mbox{ and } \overline{R}(x)=\infty\}$ (Proposition~\ref{genericpoints}).

\begin{rek} We note that there are other examples of systems $(X,T,\mu)$ for which $\overline{R}(x)>\underline{R}(x)$ for $\mu$-a.e.$\;x\in X$ (respectively, $\overline{R}(x,y)>\underline{R}(x,y)$ for $(\mu\times\mu)$-a.e.$\;(x,y)\in X\times X$); this is particularly true if: $X=\mathbb{S}^1$, $T=T_\alpha$ is the translation by $\alpha$, an irrational number of type greater than one, and $\mu$ is the Haar measure over $\mathbb{S}^1$ (see Theorem~6 in~\cite{Galatolo2011}); $(M,S,\mu)$ is a particular case of a skew-product system $M=\Omega\times\mathbb{T}^d$, whose underlying space is a Riemannian manifold, whose mapping $S$ is a skew-product such that the action over the torus is a translation by a number of finite Diophantine type, and whose measure $\mu$ is the product of an equilibrium state of some potential $\psi:\Omega\rightarrow\mathbb{R}$ with the Haar measure over $\mathbb{T}^d$  (see Sections~3 and~4 in~\cite{Galatolo2011} for details).

For these systems, it is also true that $\underline{\overline{R}}(x)\neq \underline{\overline{d}}_\mu(x)$ for $\mu$-a.e.$\;x\in X$ (respectively, $\underline{\overline{R}}(x,y)\neq \underline{\overline{d}}_\mu(y)$ for $(\mu\times\mu)$-a.e.$\;(x,y)\in X\times X$). Nevertheless, it follows from Theorem~\ref{teocentral3} that for a typical invariant measure, $\mu$, of $(X,T)$, $\underline{R}(y,y)=\underline{R}(x,y)=\underline{d}_\mu(y)=0$ for $(\mu\times\mu)$-a.e.$\;(x,y),(y,y)\in X\times X$ (respectively, $\overline{R}(y,y)=\overline{R}(x,y)=\overline{d}_\mu(y)=\infty$ for $(\mu\times\mu)$-a.e.$\;(x,y),(y,y)\in X\times X$).
\end{rek}

Finally, combining items II, VII and VIII of Theorem~\ref{teocentral3}, one concludes that for a typical measure $\mu\in\overline{\mathscr{R}}\cap\underline{\mathscr{R}}\cap C_X$, almost every $T$-orbit $\mathcal{O}(x)$ densely fills the whole space (given that $\mu$ is supported on a dense subset of $X$ and $\underline{R}(x,y)=0$ for $(\mu\times\mu)$-a.e$\;(x,y)\in X\times X$), but not in a homogeneous fashion. Namely, as in the previous analysis, there exists a time scale for which the first entrance time of $\mathcal{O}(x)$ to one of the spherical neighborhoods (which depend on time) of $y$ is of order 1; accordingly, there exists a time sequence for which the first entrance time of the $\mathcal{O}(x)$ to one of the spherical neighborhoods of $y$ increases as fast as possible. Naturally, these time scales depend on the pair $(x,y)\in X\times X$. 

\

\begin{rek} \label{MIR} 
\begin{itemize}~
\item[i)] It is true that the sets defined in items III to VIII of Theorem  \ref{teocentral3} are $G_\delta$ subsets of $\M(T)$ for any topological dynamical system $(X,T)${\footnote{By a topological dynamical system we understand a pair $(X,T)$ such that $X$ is a Polish metric space and $T:X\rightarrow X$ is a continuous transformation.}}; 
  this is particularly true for Axiom A systems on smooth compact Riemannian n-manifolds, $(M,T)$, where $f:M\rightarrow M$ is a diffeomorphism and $M$ is an $f$-invariant component of the non-wandering set of $f$ such that $T:=f\restriction M$ is topological transitive (the existence of $M$ is guaranteed by the Spectral Theorem). 

\item[ii)] It is also true that the sets defined in items III, V and VII of Theorem \ref{teocentral3} are dense in $\M(T)$ for any topological dynamical system $(X,T)$ such that 
  the set of $T$-periodic measures is dense in $\M(T)$. This is particularly true for any system satisfying the specification property (see~\cite{Sigmund1974} for a proof of this proposition and examples of systems that satisfy this property; see also \cite{Ren}), or even milder conditions (see~\cite{Gelfert,Hirayama,Kwietniak,Li,Liang} for a broader discussion involving such conditions).
  
\item[iii)] Since the Axiom A systems described in item i) also have a dense set of $T$-periodic measures (here, $X$ stands for a closed $f$-invariant set and $T:=f\restriction X$ is topologically transitive; see~\cite{Sigmund1970}), one may combine both properties and obtain the following result.
\end{itemize}

\begin{teo}
\label{AxiomA}
Let $(X,T)$ be an Axiom A system as described in items i) and iii) above. Then, the set $\{\mu\in \M_e \mid \underline{R}(x,y)=0,$ for $(\mu\times\mu)\textrm{-}a.e.\,(x,y)\in X\times X\}$ is residual in $\mathcal{M}_e$.
\end{teo}

\begin{itemize}
\item[iv)] The hypothesis that the alphabet $M$ is perfect is crucial for items  IV and VI of Theorem \ref{teocentral3}. Namely, the fact that $M$ does not have isolated points is required to guarantee that one can always choose the periodic point $x$ of period $s$ in the statement of Lemma~\ref{denseperiodic} so that $x_i\neq x_j$ if $i\neq j$, $1\leq i,j\leq s$. This result, whose proof relies on the product structure of $X$, is required in the proof of Proposition \ref{central}, which in turn guarantees that the sets presented in items IV and VI of Theorem \ref{teocentral3} are dense. Indeed, our strategy depends on the fact that the set of ergodic measures with arbitrarily large entropy is dense in $\M_e$; here, we explicitly use the fact that the lower packing dimension of an ergodic measure is lower bounded by (up to a constant) its entropy (see Lemma~\ref{dimpos01}).    
\end{itemize}
 \end{rek}

\begin{rek} It is important to note that the results stated in items IV, VI and VIII in Theorem~\ref{teocentral3} are false if $M$ is a finite alphabet and $d$ is a hyperbolic metric in $X$ (which is compatible with the product topology); such metric exists since $(X,T)$ is expansive (see~\cite{Fathi1989} for the definition of hyperbolic metric and the proof of its existence when the system is expansive). Roughly, one can say that for such metric, there exist $\ve>0$ and $k>1$ such that, for each $x\in X$, each $n\in\N$ and each $0<r<\ve/k$,
  \begin{equation}\label{dball}
    B(x,n,r)\subset B(x,rk^{-n}),
  \end{equation}
where $B(x,n,r):=\{z\in X\mid\rho(T^iz,T^ix)<\ve,\;\;\forall\; 0\le i\le n-1\}$ stands for the $n$-th dynamical ball of radius $r$ centered at $x$. 

Namely, the following statements are true.  

  \begin{enumerate}
  \item One has that, by the dual argument to the one presented in the proof of Lemma~\ref{dimpos} and by~\eqref{dball}, $\overline{d}_\mu(x)\le h_\mu(T)/k$ for $\mu$-a.e $x\in X$, where $h_\mu(T)$ stands for the metric entropy of $\mu\in\M(T)$. Combining this with Proposition~\ref{BGT}, it follows that for each $\mu\in\M(T)$, $\dim_P^+(\mu)\le h_\mu(T)/k$.
\item It follows from~\eqref{dball}, Theorem~A and Proposition~A in~\cite{Varandas} that if $\mu\in\M_e$, then $\overline{R}(x)\le h_\mu(T)/k$ for $\mu$-a.e $x\in X$.
\item  It follows from~\eqref{dball} and adapted versions of Theorem~A and Proposition~A in~\cite{Varandas} that if $\mu\in\M_e$, then $\overline{R}(x,y)\le h_\mu(T)/k$ for $(\mu\times\mu)$-a.e $(x,y)\in X\times X$.
\end{enumerate}

Combining items 1, 2 and 3 with the fact that $\{\mu\in\M_e\mid h_\mu(T)=0\}$ is a generic subset of $\M(T)$ (see~\cite{Sigmund1974} for details), it follows that each of the sets $\{\mu\in\M(T)\mid\dim_P(\mu)=0\}$, $\{\mu\in\M(T)\mid R(x)=0,$ for $\mu\textrm{-a.e.}\;x\}$ and $\{\mu\in\M(T)\mid R(x,y)=0,$ for $(\mu\times\mu)\textrm{-a.e.}\;(x,y)\}$ is a generic subset of $\M(T)$, at least for a hyperbolic metric.

  So, in terms of the packing dimension, the upper recurrence rate and the upper quantitative waiting time indicator of a typical invariant measure, there is a striking difference between the full-shift defined over finite and perfect (uncountable) alphabets. Note that $(X,T)$ is not expansive if $M$ is a perfect alphabet, and that the metric defined by~\eqref{metric} is not hyperbolic (by Theorem~5.3 in~\cite{Fathi1989}, given that $\dim_{top}(X)=+\infty$ in this case).

  We emphasize again that in order to prove Theorem~\ref{teocentral3},  we use the fact that $M$ does not have isolated points. Since this is false if $M$ is a countable compact alphabet, and since $(X,T)$ is not expansive (by Hedlund-Reddy's Theorem; see~\cite{Downarowicz}), it is not clear in this case which are the values of the packing dimension, the upper recurrence rate and the upper quantitative waiting time indicator of a typical invariant measure.
\end{rek}

The paper is organized as follows. In Section~\ref{TVI-VII}, we present several results used in the proof of items III and IV of Theorem~\ref{teocentral3}. Section~\ref{TVIII-IX} is devoted to the proof of items V-VIII of Theorem~\ref{teocentral3}, and as in Section~\ref{TVI-VII}, we prove some auxiliary results which are the counterparts, for the returning rates and the first entrance rates, of the results stated in Section~\ref{TVI-VII}. In Appendix, we present the proof of Proposition~\ref{BGT}.

\section{Sets of ergodic measures with zero Hausdorff and infinity packing dimensions}
\label{TVI-VII}

\subsection{$HD$ and $PD$ are $G_\delta$ sets}

In what follows, $\M(X)$ denotes the set of Borel probability measures defined on the separable metric space $(X,d)$, endowed with the weak topology. Note that if $X$ is Polish (respectively, compact), then $\M(X)$ is also Polish (respectively, compact); see~\cite{Parthasarathy1961}.

We prove here that $\{\mu\in\M(X)\mid \dim_H(\mu)=0\}$ and, for each $\alpha>0$, $\{\mu\in\M(X)\mid \dim_P^-(\mu)\ge \alpha\}$ are both $G_\delta$ subsets of $\M(X)$; this implies that the sets $HD$ and $PD$ defined in the statement of Theorem~\ref{teocentral3} are also $G_\delta$ subsets of $\M(T)$.

\begin{lema}
  \label{asympt1}
  Let $(X,d)$ be a Polish metric space and let 
  $\mu\in\M(X)$. 
  Then, for each $x\in X$, 
\[\underline{d}_{\mu}(x)=\liminf_{\ve\to 0}\frac{\log f_{x,\ve}(\mu)}{\log \ve},\qquad\overline{d}_{\mu}(x)=\limsup_{\ve\to 0}\frac{\log f_{x,\ve}(\mu)}{\log \ve},\]
where, for each $x\in X$ and each $\ve>0$, $$f_{x,\ve}(\,\cdot\,):\M(X)\rightarrow [0,1] \mbox{ ~is defined by the law~ } f_{x,\ve}(\mu):=\int f^{\ve}_x(y)d\mu(y),$$
 and  
 $f^{\ve}_x:X\rightarrow[0,1]$ is defined by the law 
 $$
f^{\ve}_x(y):= \left\{ \begin{array}{lcc}
                    
            1 & ,if  & d(x,y) \leq \ve, \\
            \\ -\dfrac{d(x,y)}{\ve}+2&, if & \ve\le d(x,y)\le 2\ve,\\
            \\ 0 &   ,if  & d(x,y) \geq 2\ve.
          \end{array}
\right.$$
Furthermore, the function $f_{\ve}(\mu,x)=f_{x,\ve}(\mu):\M(X)\times X\rightarrow [0,1]$ is jointly continuous. 
\end{lema}
\begin{proof}
It follows from the definition of $f^\ve_{x}$ that, for each $x\in X$ and each $\ve>0$, $f_{x,\ve/2}(\mu)\leq \mu(B(x,\ve))\leq f_{x,2\ve}(\mu).$ Then, if $\mu(B(x;\ve))>0$, one has $\frac{\log f_{x,\ve/2}(\mu)}{\log \ve}\geq \frac{\log \mu(B(x,\ve))}{\log \ve}\geq \frac{\log f_{x,2\ve}(\mu)}{\log \ve}$, which proves the first assertion. If $\mu(B(x;\ve))=0$, given that $f_{x,\ve/2}(\mu)\leq \mu(B(x,\ve))$, just set $\limsup_{\ve\to 0}(\inff)\frac{\log f_{x,\ve}(\mu)}{\log \ve}=+\infty$.

Note that, for each $x\in X$ and each $\ve>0$, $f^{\ve}_x: X\rightarrow \R$ is a continuous function such that, for each $y\in X$, $ \chi_{_{B(x,\ve/2)}}(y) \leq f_{x}^{\ve}(y)\leq \chi_{_{B(x,2\ve)}}(y)$. Given that $f_x^\ve(y)$ depends only on $d(x,y)$, it is straightforward to show that $f^{\ve}_{x_n}$ converges uniformly to $f^{\ve}_x$ on $X$ when $x_n\rightarrow x$.

We combine this remark with Theorems 2.13 and 2.15 in \cite{Habil} in order to prove that $f_\ve(\mu,x)$ is jointly continuous. Let $(\mu_m)$ and $(x_n)$ be sequences in $\M(X)$ and $X$, respectively, such that $\mu_m\to \mu$ and  $x_n\to x$. Firstly, we show that 
$$\lim_{m\to \infty}\lim_{n\to\infty} f_{\ve}(\mu_m,x_n)=\lim_{m\to \infty}\lim_{n\to\infty}\int f^{\ve}_{x_n}(y)d\mu_m(y)=f_{\ve}(\mu,x).$$
 
Since, for each $y\in X$, $|f^{\ve}_{x_n}(y)|\leq 1$, it follows from dominated convergence that, for each $m\in\mathbb{N}$, $\lim_{n\to\infty}\int f^{\ve}_{x_n}(y)d\mu_m(y)= \int f^{\ve}_{x}(y)d\mu_m(y)$. 
Now, since $f_x^\ve$ is continuous, it follows from the the definition of weak convergence that
\begin{equation*}
\lim_{m\to \infty}\lim_{n\to\infty}\int f^{\ve}_{x_n}(y)d\mu_m(y)= \lim_{m\to \infty}\int f^{\ve}_{x}(y)d\mu_m(y)=f_{\ve}(\mu,x).
\end{equation*}

The next step consists in showing that, for each $n\in \mathbb{N}$, the function $\f_n:\mathbb{N}\rightarrow \mathbb{R}$, defined by the law $\f_n(m):= f_{\ve}(\mu_m,x_n)$,  converges  uniformly in $m\in\N$ to $\f(m):=\lim_{n\to \infty}f_{\ve}(\mu_m,x_n)=\int f^{\ve}_{x}(y)d\mu_m(y)$. Let $\delta>0$. 
Since $f^{\ve}_{x_n}(y)$ converges uniformly to $f^{\ve}_{x}(y)$, there exists $N\in \N$ such that, for each $n\ge N$ and each $y\in X$, $\left| f^{\ve}_{x_n}(y)- f^{\ve}_{x}(y)\right|<\delta$. Then one has, for each $n\geq N$ and each $m\in\N$, 
\begin{eqnarray*}
|\f_n(m)-\f(m)|=\left| \int f^{\ve}_{x_n}(y)d\mu_m(y)- \int f^{\ve}_{x}(y)d\mu_m(y) \right|
&\leq & \int \left| f^{\ve}_{x_n}(y)- f^{\ve}_{x}(y)\right| d\mu_m(y)\\
&<& \delta.
\end{eqnarray*}

It follows from Theorem  2.15 in \cite{Habil} that $\lim_{n,m \to \infty} f_{\ve}(\mu_m,x_n)= f_{\ve}(\mu,x)$. Given that $\lim_{n\to\infty} f_{\ve}(\mu_m,x_n)=\int f^{\ve}_{x}(y)d\mu_m(y)$ and that $\lim_{m\to\infty} f_{\ve}(\mu_m,x_n)=\int f^{\ve}_{x_n}(y)d\mu(y)$ exist for each $m\in \N$ and each $n\in\N$,  respectively, Theorem 2.13 in \cite{Habil} implies that 
$$\lim_{m\to \infty}\lim_{n\to\infty} f_{\ve}(\mu_m,x_n)=\lim_{n\to \infty}\lim_{m\to\infty} f_{\ve}(\mu_m,x_n)=\lim_{n,m \to \infty} f_{\ve}(\mu_m,x_n)=f_{\ve}(\mu,x).$$ 

Hence, if $(\mu_n,x_n)$ is some sequence in $\M(X)\times X$ (endowed with the product topology) such that $(\mu_n,x_n)\to (\mu,x)$, then $\lim_{n \to \infty} f_{\ve}(\mu_n,x_n)=f_{\ve}(\mu,x)$, showing that $f_{\ve}(\,\cdot\,,\,\cdot\,)=f_{\,\cdot\,,\ve}(\,\cdot\,)$ is jointly continuous at $(\mu,x)$.
\end{proof}



For each $t>0$, let $\ve=1/t$. Since, for each $x\in X$, 
\[\underline{d}_{\mu}(x)=\lim_{s\to \infty}\inf_{t\geq s} \frac{\log f_{x,1/t}(\mu)}{-\log t},\qquad \overline{d}_{\mu}(x)=\lim_{s\to \infty}\sup_{t\geq s} \frac{\log f_{x,1/t}(\mu)}{-\log t},\]
we set, for each $s\in\mathbb{N}$, 
\begin{eqnarray*}\label{beta}
\overline{\beta}_{\mu}(x,s)=\sup_{t> s} \frac{\log f_{x,1/t}(\mu)}{-\log t}~~\mbox{ and } ~~ \underline{\beta}_{\mu}(x,s)=\inf_{t> s} \frac{\log f_{x,1/t}(\mu)}{-\log t};
\end{eqnarray*}
note that, for each $x\in X$, $\N\ni s\mapsto\overline{\beta}_{\mu}(x,s)\in[0,+\infty]$ is non-increasing, whereas $\N\ni s\mapsto\underline{\beta}_{\mu}(x,s)\in[0,+\infty]$ is a non-decreasing function.

\begin{propo}
\label{Gdelta3} 
Let $(X,d)$ be a Polish metric space and let $\alpha>0$. Then, each of the sets  
  \[PD(\alpha)=\{ \mu\in \M(X) \mid\dim_P^-(\mu)\geq \alpha\},\]
  \[HD=\{ \mu\in \M(X) \mid\dim_H(\mu)=0\}\]
is a $G_{\delta}$ subset of $\M(X)$.
\end{propo}

\begin{proof}
  Since the arguments in both proofs are similar, we just prove the statement for $PD(\alpha)$. We show that $\M(X)\setminus PD(\alpha)$ is an $F_{\sigma}$ set.

  {\textit{Claim 1.}}  $PD(\alpha)=\bigcap_{s\in\N}\{\mu\in \M(X) \mid \mu\textrm{-}\essinf\overline{\beta}_\mu(x,s)\ge\alpha\}$.

Let $\mu\in PD(\alpha)$.  Since, for each $x\in X$,  $\N\ni s\mapsto\overline{\beta}_\mu(x,s)\in[0,\infty]$ is a non-increasing function, it follows that, for each $s\in\N$, $\mu\textrm{-}\essinf\overline{\beta}_\mu(x,s)\geq \alpha$.

Now, let $\mu\in\bigcap_{s\in\N}\{\nu\in \M(X) \mid \nu\textrm{-}\essinf\overline{\beta}_\nu(x,s)\ge\alpha\}$. Then, for each $s\in\N$, there exists a measurable $A_s\subset X$ with $\mu(A_s)=1$, such that for each $x\in A_s$, $\overline{\beta}_\mu(x,s)\ge \alpha$. Let $A:=\bigcap_{s\ge 1}A_s$; then, for each $x\in A$, one has $\overline{d}_\mu(x)=\lim_{s\to\infty}\overline{\beta}_\mu(x,s)\geq\alpha$; since $\mu(A)=1$, the result follows from Proposition~\ref{BGT}.

Let $\mu \in \M(X)$, let $k,s\in \N$, set $Z_{\mu}(s,k)=\{x\in X\mid \overline{\beta}_{\mu}(x,s)\leq\alpha-1/k\}$ and set, for each $l\in\N$,
\begin{eqnarray*}
\M_{s,k}(l)=\{\nu\in\M(X)\mid \nu(Z_\nu(s,k))\ge 1/l\}.
\end{eqnarray*}

\textit{Claim 2.} $Z_{\mu}(s,k)$ is  closed.

Let $(z_n)$ be a sequence in $Z_{\mu}(s,k)$ such that $z_n\to z$, and let $t>0$. 
Since for each fixed $\mu\in\M(X)$, the mapping $X\ni x\mapsto f_{x,1/t}(\mu)\in(0,1]$ is continuous (see the proof of Lemma~\ref{asympt1}), the mapping $X\ni x\longmapsto  \overline{\beta}_{\mu}(x,s)\in[0,+\infty)$ is  lower semi-continuous, which implies that $z\in Z_{\mu}(s,k)$. 

\textit{Claim 3.} $W_{s,k}=\{(\mu,x)\in \M(X)\times X\mid \overline{\beta}_{\mu}(x,s)>\alpha-1/k\}$ is open.

This is a consequence of the fact that, by Lemma~\ref{asympt1}, the mapping $\M(X)\times X\ni(\mu,x)\longmapsto \overline{\beta}_{\mu}(x,s)$
is lower semi-continuous. 

\

Now, we show that $\M_{s,k}(l)$ is closed. Let $(\mu_n)$ be a sequence in $\M_{s,k}(l)$ such that $\mu_n\to \mu$. Suppose, by absurd, that $\mu\notin \M_{s,k}(l)$; we will find that $\mu_n\notin \M_{s,k}(l)$ for~$n$ sufficiently large, a contradiction.

If $\mu\notin \M_{s,k}(l)$, then $\mu(A)>1-1/l$ where, $A=X\setminus Z_{\mu}(s,k)$.  Given that $\mu$ is tight ($\mu$ is a probability Borel measure and the space $X$ is Polish), there exists a compact set $C\subset A$ such that $\mu(C)>1-1/l$.

The idea is to construct a suitable subset of $W_{s,k}$ that contains a neighborhood of $\{\mu\}\times C$. Let, for each $x\in C$, $V_x\subset W_{s,k}$ be an open neighborhood of~$(\mu,x)$ (such open set exists, by Claim 3; that is, $V_x:=B((\mu,x);\varepsilon)=\{(\nu,y)\in\M(X)\times X\mid\max\{\rho(\nu,\mu),d(x,y)\}<\varepsilon\}$, for some suitable $\varepsilon>0$ (where $\rho$ is any metric defined in $\M(X)$ which is compatible with the weak topology); then, $\{V_x\}_{x\in C}$ is an open cover of $\{\mu\}\times C$, and since $\{\mu\}\times C$ is a compact subset of $\M(X)\times X$, it follows that one can extract from $\{V_x\}_{x\in C}$ a finite sub-cover, $\{V_{x_i}\}_{i=1}^n$.

We affirm that there exists an $\ell\in\mathbb{N}$ (which depends on $C$) such that $\{\mu_n\}_{n\ge\ell}\subset\bigcap_{i}(\pi_1(V_{x_i}))$. Namely, for each $i$, there exists an $\ell_i$ such that $\{\mu_n\}_{n\ge\ell_i}\subset\pi_1(V_{x_i})$; set $\ell:=\max\{\ell_i\mid i\in\{1,\ldots,n\}\}$, and note that for each $i$, $\{\mu_n\}_{n\ge\ell}\subset\pi_1(V_{x_i})$. Set also $\mathcal{I}:=\bigcap_{i}(\pi_1(V_{x_i}))$ and $\mathcal{O}:=\bigcup_{i}(\pi_2(V_{x_i}))$.

Since for each $i$, $V_{x_i}=\pi_1(V_{x_i})\times\pi_2(V_{x_i})$,  and given that
\begin{eqnarray*}
\{\mu_n\}_{n\ge\ell}\times\mathcal{O}&\subset& \mathcal{I}\times\mathcal{O}= 
\bigcup_{j} \left( \left[\bigcap_{i}\pi_1(V_{x_i}) \right] \times \pi_2(V_{x_j})\right)
\subset  \bigcup_{j}(\pi_1(V_{x_j})\times\pi_2(V_{x_j}))\\
&=&\bigcup_{j}V_{x_j}\subset W_{s,k},
\end{eqnarray*}
it follows that, for each $n\ge\ell$ and each $y\in\mathcal{O}$, $\overline{\beta}_{\mu_n}(y,s)>\alpha-1/k$. Moreover, $\mathcal{O}$ is an open set that contains $C$.

On the other hand, weak convergence implies that 
\[
\displaystyle\limsup_{n\to \infty}\mu_n(X\setminus\mathcal{O})\leq\mu(X\setminus\mathcal O)\leq  \mu(X\setminus C)<\frac{1}{l},
\] from which follows that there exists an $\tilde{\ell}\ge\ell$ such that, for $n\ge\tilde{\ell}$, $\mu_n(X\setminus \mathcal{O})<1/l$.

Combining the last results, one concludes that for  $n\ge\tilde{\ell}$, $\mu_n(X\setminus \mathcal{O})<1/l$, and for each $y\in \mathcal{O}$, $\overline{\beta}_{\mu_n}(y,s)>\alpha-1/k$, so
\[\mu_n(Z_{\mu_n}(s,k))\le\mu_n(X\setminus\mathcal{O})<\frac{1}{l};\]
this contradicts the fact that, for each $n\in\mathbb{N}$, $\mu_n\in \mathcal{M}_{s,k}(l)$. Hence, $\mu \in  \mathcal{M}_{s,k}(l)$, and $\mathcal{M}_{s,k}(l)$ is a closed subset of $\mathcal{M}(X)$. 

Finally, it follows  $\M(X)\setminus PD(\alpha)=\bigcup_{s\in \N}\bigcup_{k\in \N}\bigcup_{l\in \N}\M_{s,k}(l)$ is an $F_{\sigma}$ subset of $\M(X)$, concluding the proof of the lemma.
\end{proof}

\subsection{$HD$ and $PD$ are dense sets}
\label{Subdense}

Here we prove that the sets $HD$ and $PD$ defined in the statement of Theorem~\ref{teocentral3} are dense in $\M(T)$. In order to prove that $PD$ is dense, we will use the fact that $T$ is a bi-Lipschitz function.

So, from now on, we assume that $(X,d)$ is a Polish metric space and that $T:X\rightarrow X$ is a Lipschitz function, with Lipschitz constant $\Lambda>1$. Assume also that $T^{-1}:X\rightarrow X$ exists as a Lipschitz function, with Lipschitz constant $\Lambda^\prime>1$.

\begin{propo}
\label{uniforminvariance}
Let $\mu\in \M(T)$.  Then, for each $x\in X$, $\overline{d}_{\mu}(x)=\overline{d}_{\mu}(Tx)$, $\underline{d}_{\mu}(x)=\underline{d}_{\mu}(Tx)$. 
\end{propo}
\begin{proof}
 It follows from  Birkhoff's Ergodic Theorem that, for each $z\in X$ and each $\ve>0,$ the limit 
\begin{eqnarray}\label{IB}
\tilde{\varphi}_{_{B(z,\ve)}}(y)=\lim_{n\to\infty}\frac{1}{n} \sum_{i=0}^{n-1}f^\ve_z(T^i(y))
\end{eqnarray}
exists for $\mu$-a.e.$\;y\in X$, and 
\begin{eqnarray*}
\int \tilde{\varphi}_{_{B(z,\ve)}}(y) d\mu(y)=\int  f^\ve_z(y)d\mu(y)=f_{z,\ve}(\mu).
\end{eqnarray*}

Fix $x\in \supp(\mu)$. It is straightforward to show that, for each $y\in X$ and each $i\in\N\cup\{0\}$, one has $f_x^{\ve/\Lambda}(T^i(y))\le f_{Tx}^{\ve}(T^{i+1}(y))$. Letting $z=x$ and $z=Tx$ in~\eqref{IB}, respectively, one gets $\tilde{\varphi}_{_{B(x,\ve/\Lambda)}}(y)\leq \tilde{\varphi}_{_{B(Tx,\ve)}}(y)$ for $\mu$-a.e.$\;y\in X$, from which follows that $f_{x,\ve/\Lambda}(\mu)\leq f_{Tx,\ve}(\mu)$.

{\textit{Case 1: $x\in\supp(\mu)$}}. Note that, for each $\eta>0$, $f_{x,\eta}(\mu)>0$. Let $\ve=1/t$, $t=l/\Lambda$  and $s\geq 1+1/\Lambda$; then,
\begin{eqnarray*}
  \sup_{t\geq s}\frac{\log f_{Tx,1/t}(\mu)}{-\log t}&\leq& \sup_{l\geq \Lambda s}\frac{\log l}{\log l - \log \Lambda}\frac{\log f_{x,1/l}(\mu)}{-\log l}\leq \frac{\log (\Lambda s)}{\log (\Lambda s) - \log\Lambda} \sup_{l\geq \Lambda s}\frac{\log f_{x,1/l}(\mu)}{-\log l}\\
  &=& A_{\Lambda}(s) \sup_{l\geq \Lambda s}\frac{\log f_{x,1/l}(\mu)}{-\log l},
\end{eqnarray*}
where $A_{\Lambda}(s):=\frac{\log s + \log\Lambda}{\log s}$ (since $s\geq 1+1/\Lambda$, one has $l\geq \Lambda+1$).

Using the same idea, one can prove that 
$f_{z,\ve/\Lambda^\prime}(\mu)\leq f_{T^{-1}z,\ve}(\mu)$; letting $z=Tx$, one gets $f_{Tx,\ve/\Lambda^\prime}(\mu)\leq f_{x,\ve}(\mu)$. Thus, the previous discussion leads to 
\begin{eqnarray*}
\label{invariance}
\overline{\beta}_{\mu}(Tx,s)\leq A_{\Lambda}(s) \overline{\beta}_{\mu}(x,\Lambda s) ~\mbox{ and }  ~  \overline{\beta}_{\mu}(x,s)\leq A_{\Lambda^\prime}(s)\overline{\beta}_{\mu}(Tx,\Lambda^\prime s);
\end{eqnarray*}
one can combine these inequalities and obtain, for each $x\in X$ and each $s\ge\max\{1+1/\Lambda,1+1/\Lambda^\prime\}$,
\begin{eqnarray*}
\overline{\beta}_{\mu}(Tx,s)\leq A_{\Lambda}(s) \overline{\beta}_{\mu}(x,\Lambda s)\leq A_{\Lambda}(s)A_{\Lambda^\prime}(\Lambda s) \overline{\beta}_{\mu}(Tx,\Lambda\cdot\Lambda^\prime s).
\end{eqnarray*}

Now, taking the limit $s\to\infty$ in the inequalities above and observing that $A_{\Lambda}(s)$ and $A_{\Lambda^\prime}(s)$ are decreasing functions such that $\lim_{s\to\infty}A_{\Lambda(\Lambda^\prime)}(s)=1$, one gets $\overline{d}_{\mu}(Tx) =\overline{d}_{\mu}(x)$.

{\textit{Case 2: $x\notin\supp(\mu)$.}} It follows from the $T$-invariance of $\supp(\mu)$ that $T(x)\notin\supp(\mu)$; thus, $\overline{d}_\mu(Tx)=+\infty=\overline{d}_\mu(x)$.

The proof that, for each $x\in X$ $\underline{d}_{\mu}(Tx) =\underline{d}_{\mu}(x)$, is analogous; therefore, we omit it.
\end{proof}

\begin{rek}
\label{forergodic}
If $\mu\in \M_e$, since $T$ and $T^{-1}$ are Lipschitz functions,  it follows from Proposition~\ref{uniforminvariance} that $\overline{d}_{\mu}(x)$ and $\underline{d}_{\mu}(x)$ are constants $\mu$-a.e. (an analogous result can be found in \cite{Cutler1993}, Theorem 4.1.10 chapter 1). 
\end{rek}

\begin{lema}
\label{dimpos}
Let $X$ be a compact metric space. If $\mu\in\M(T)$, then $\dim_P^+(\mu)\geq \frac{h_{\mu}(T)}{\log \Lambda}$. 
\end{lema}
\begin{proof}
Fix $x\in X$,  $n\geq1$ and $\ve>0$. Given $y\in B(x,\ve \Lambda^{-n})$, one has, for each $0\le i\le n$, $\rho(T^iy,T^ix)\le\Lambda^i\rho(x,y)\le\Lambda^{i-n}\ve<\ve$, which shows that $y\in B(x,n,\ve)=\{z\in X\mid\rho(T^iz,T^ix)<\ve,\;\;\forall\; 0\le i\le n-1\}$. Hence, for each $x\in X$ and each $\ve>0$,
\begin{eqnarray*}
\overline{d}_{\mu}(x)  \ge   \limsup_{n\to \infty} \frac{\log\mu(B(x,\ve \Lambda^{-n}))}{\log \ve \Lambda^{-n}}  
&\geq & \limsup_{n\to \infty} \frac{\log\mu(B(x,n,\ve))}{-n} ~\frac{1}{\frac{-\log\ve}{n}+\log \Lambda}\\ 
&\geq & \limsup_{n\to \infty} \frac{\log\mu(B(x,n,\ve))}{-n} \frac{1}{\log\Lambda};
\end{eqnarray*}
it follows that, for $\mu$-a.e.$\;x\in X$,
\begin{eqnarray}
\label{entropylocal}
\overline{d}_{\mu}(x)  \ge   \lim_{\ve\to 0}\limsup_{n\to \infty} \frac{\log\mu(B(y,\ve \Lambda^{-n}))}{\log \ve \Lambda^{-n}}  
&\geq & \lim_{\ve\to 0} \limsup_{n\to \infty} \frac{\log\mu(B(x,n,\ve))}{-n} \frac{1}{\log\Lambda} \nonumber \\
&=& h_{\mu}(T,x)\frac{1}{\log\Lambda},
\end{eqnarray}
where $h_{\mu}(T,x)$ is the so-called local entropy of $\mu$ at $x$. One also has, using Brin-Katok Theorem, that $\int h_{\mu}(T,x) d\mu(x)= h_{\mu}(T)$ (the compacity of $X$ is required in this step; see \cite{BrinKatok}). Hence, there exists a measurable set $B$, with $\mu(B)>0$, such that, for each $x\in B$,  $\overline{d}_{\mu}(x)\geq \frac{h_{\mu}(T)}{\log \Lambda}$. The result is now a consequence of Proposition~\ref{BGT}.
\end{proof}

\begin{lema}
\label{dimpos01}
Let $X$ be a Polish metric space and let $\mu\in\M_e$. Then, $\dim_P^-(\mu)\geq \frac{h_{\mu}(T)}{\log \Lambda}$. 
\end{lema}
\begin{proof}
  Since  $\mu$ is ergodic, it follows from Proposition \ref{uniforminvariance} that $\overline{d}_{\mu}(x)$ is constant for $\mu$-a.e.$\,x$ (this constant may be infinite).
  
One also has, by Lemma~2.8 in~\cite{Riquelme}, that $\int \underline{h}_{\mu}(T, x)d\mu(x)=\mu\textrm{-}\essinf \underline{h}_{\mu}(T, x)$, and then, by Theorem~2.9 in \cite{Riquelme}, that $\int \underline{h}_{\mu}(T, x)d\mu(x)\ge h_\mu(T)$.  Thus, by inequality~\eqref{entropylocal}, one gets, for $\mu$-a.e.$\;x\in X$,
\[\overline{d}_{\mu}(x)=\int \overline{d}_{\mu}(x)d\mu(x)\ge \int \underline{h}_{\mu}(T, x)\frac{1}{\log\Lambda}d\mu(x)\ge h_\mu(T)\frac{1}{\log\Lambda}.\] 
The result is obtained again by an application of Proposition~\ref{BGT}.
\end{proof}

Now, we return to the specific setting where $X=\prod_{-\infty}^{+\infty}M$, with $M$ a perfect and compact metric space, with $\rho$ given by~\eqref{metric} and with $T$ the full-shift over $X$.

The next result is an extension of Lemma 6 in~\cite{Sigmund1971} to $X=\prod_{-\infty}^{+\infty} M$, where $M$ is perfect and compact (the hypothesis of $M$ being perfect is required to guarantee that one can always choose the periodic point $x$ in the statement of Lemma~\ref{denseperiodic} in such way that $x_i\neq x_j$ if $i\neq j$, $i,j=1,\ldots,s$; see Remark~3.4 in~\cite{Oxtoby1963}). 

\begin{lema}[Lemma 6 in \cite{Sigmund1971}]
\label{denseperiodic}
Let $X=\prod_{-\infty}^{+\infty}M$, where $M$ is perfect and compact, let $\mu\in\M(T)$ and let $s_0>0$. Then, $\mu$ can be approximated by a $T$-periodic measure $\mu_x\in \M_e$ such that $x\in X$ has period $s\geq s_0$ and $x_i\neq x_j$ if $i\neq j$, $i,j=1,\ldots,s$.
\end{lema}


The next result is an extension of Theorem~2 in~\cite{Sigmund1971} (which is proved using Lemma~\ref{denseperiodic}) to the space $X=\prod_{-\infty}^{+\infty} M$, with $M$ a perfect Polish space (the original result was proved for $M=\mathbb{R}$); 
we leave the details to the dedicated reader.

\begin{propo}[Theorem~2 in \cite{Sigmund1971}]
\label{central}
Let $L>0$. Then, $\{\mu\in \M_e\mid h_\mu(T)\ge L\}$ is a dense subset of $\mathcal{M}_e$.
\end{propo}


\begin{propo}
\label{densepack}
Let $L>0$. Then, $\{\mu\in \M_e\mid\dim_P^-(\mu)\ge L\}$ is a dense subset of $\mathcal{M}_e$.
\end{propo}
\begin{proof}
  Let $\delta>0$, and let $\mu\in\M_e$. It is straightforward to show that $T$ is a Lipschitz function with constant $\Lambda=2$.  Set $K:=L\log 2$. It follows from Proposition~\ref{central} (see Lemma 7 in~\cite{Sigmund1971}) that given any neighborhood of $\mu$ (in the induced topology) of the form $V_\mu(f_1,\ldots,f_r;\delta)=\{\nu\in\M_e\mid|\int f_i\,d\nu-\int f_i\,d\mu|<\delta$, $i=1,\ldots,r\}$ (where $\delta>0$ and each $f_i:X\rightarrow\mathbb{C}$ is continuous and bounded; such sets form a sub-basis of the weak topology), there exists a measure $\zeta\in V_\mu(f_1,\ldots,f_r;\delta)\cap\M_e$ such that $h_\zeta(T)\geq K$.  Now, by Lemma \ref{dimpos01}, one has $\dim_P^-(\zeta)\geq \frac{h_{\zeta}(T)}{\log 2}\ge K \frac{1}{\log2}=L$.
\end{proof}


\begin{propo}
\label{densehausdorff}
Let $M$ be a Polish metric space. Then, $\{\mu\in \M_e\mid\dim_{H}(\mu)=0\}$ is a dense subset of $\M_e$.    
\end{propo}
\begin{proof}
Let $\mu$ be the $T$-periodic measure associated with the $T$-periodic point $x\in X$, and denote its period by $k$. Naturally, $\mu(\cdot)=\frac{1}{k}\sum_{i=0}^{k-1}\delta_{f^i(x)}(\cdot)$, and for each  $i=0,\cdots,k-1$,  one has
\[\overline{d}_{\mu}(T^i(x))=\limsup_{r\to 0} \frac{\log\mu(B(T^i(x),r))}{\log r}= \limsup_{r\to 0} \frac{-\log k}{\log r}=0.\]

 The result follows now from the fact that the set of $T$-periodic measures is dense in $\M_e$ (see Theorem~3.3 in~\cite{Parthasarathy1961}).
\end{proof}

\begin{rek} The result stated in Proposition~\ref{densehausdorff} is valid for any topological dynamical system $(X,T)$ such that the set of $T$-periodic measures is dense in $\M_e$; this is particularly true for systems which satisfy the specification property (see Remark~\ref{MIR} for more details).
  \end{rek}

\begin{proof3}
\begin{itemize}
\item[\textbf{III.}] Note that, by Propositions~\ref{Gdelta3},~\ref{densepack} and item~I of Theorem~\ref{teocentral3}, $PD=\bigcup_{L\ge 1}PD(L)$ is a countable intersection of dense $G_{\delta}$ subsets of $\M_e$. 

\item[\textbf{IV.}] It follows from Propositions~\ref{Gdelta3} and~\ref{densehausdorff} that $HD$ is a dense $G_\delta$ subset of $\M_e$. The result is now a consequence of item~I in~Theorem~\ref{teocentral3}. 
\end{itemize}
\end{proof3}

The next statement says that each $\mu\in HD\cap PD\cap C_X$ is supported on a dense $G_{\delta}$ subset of $X$.

\begin{propo}
\label{localgenericpoints}
Let $\mu\in PD\cap HD\cap C_X$. Then, each of the sets $\overline{\mathfrak{D}}_{\mu}=\{x\in X\mid \overline{d}_{\mu}(x)=\infty\}$ and $\underline{\mathfrak{D}}_{\mu}=\{x\in X\mid  \underline{d}_{\mu}(x)=0\}$ is a  dense $G_{\delta}$ subset of $X$.
\end{propo}
\begin{proof}  We just present the proof that $\overline{\mathfrak{D}}_{\mu}$ is a dense $G_{\delta}$ subset of $X$. 

\emph{$\overline{\mathfrak{D}}_{\mu}$ is a $G_{\delta}$ set in X.} Let  $\alpha>0$, let $s\in \N$, and set $Z_{\mu,s}(\alpha)=\{x\in X\mid \overline{\beta_\mu}(x,s)\leq\alpha\}$. Following the proof of Claim 2 in Proposition \ref{Gdelta3}, it is clear that $Z_{\mu,s}(\alpha)$ is  closed. Thus, taking $\alpha=n\in\N$, it follows that $\overline{\mathfrak{D}}_{\mu}=\bigcap_{n\in \N}\bigcap_{s\in \N} (X\setminus Z_{\mu,s}(n))$ is  a $G_{\delta}$ subset of $X$.

\emph{$\overline{\mathfrak{D}}_{\mu}$ is dense in $X$}. Since $\mu\in PD$, one has $\mu(\overline{\mathfrak{D}}_{\mu})=1$. Suppose that $\overline{\mathfrak{D}}_{\mu}$  is not dense; then, there exist $x\in X$ and $\ve>0$ such that $B(x,\ve)\cap \overline{\mathfrak{D}}_{\mu}=\emptyset$. This implies that  $1=\mu(\overline{\mathfrak{D}}_{\mu}) + \mu(B(x,\ve))$, which is an  absurd, since $\mu(B(x,\ve))>0$ (recall that $\supp(\mu)=X$, given that $\mu\in C_X$).
\end{proof}

\section{Sets of ergodic measures with almost everywhere zero lower and infinity upper rates of recurrence and quantitative waiting time indicators}
\label{TVIII-IX}

This section presents the counterparts, for $\overline{R}(x)$, $\underline{R}(x)$, $\overline{R}(x,y)$ and $\underline{R}(x,y)$, of the results presented in Section~\ref{TVI-VII}. They are technically easier to prove than the previous ones.

Let, for each $x\in X$ and each $s\in\N$,
\[\overline{\gamma}(x,s):=\sup_{t>s}  \frac{\log \tau_{1/t}(x)}{\log t},\qquad\underline{\gamma}(x,s):=\inf_{t>s}  \frac{\log \tau_{1/t}(x)}{\log t}.\]

\begin{propo}
\label{Gdelta4} 
Let $(X,T)$ be a topological dynamical system and let $\alpha>0$. Then, each of the sets 
\begin{eqnarray*}
\overline{\mathcal{R}}(\alpha)&=&\{ \mu\in \M(T) \mid\mu\textrm{-}\essinf \overline{R}(x)\ge \alpha\},\\
\underline{\mathcal{R}}&=&\{ \mu\in \M(T) \mid \mu\textrm{-}\esssup \underline{R}(x)=0\}
\end{eqnarray*}
is $G_{\delta}$ subset of $\M(T)$.
\end{propo}
\begin{proof}
  Since the arguments in both proofs are similar, we just prove the statement for $\overline{\mathcal{R}}(\alpha)$. We show that $\M(T)\setminus\overline{\mathcal{R}}(\alpha)$ is an $F_{\sigma}$ set. We begin noting that $\overline{\mathcal{R}}(\alpha)=\bigcap_{s\in\N}\{ \mu\in \M(T) \mid \mu\textrm{-}\essinf \overline{\gamma}(x,s)\ge \alpha\}$. 
  
  Let $l,s\in \N$, set $Z_{s,l}=\{x\in X\mid \overline{\gamma}(x,s)\leq\alpha-1/l\}$,  and set for each $k\in\N$,
\begin{eqnarray*}
\M_{s,l}(k)=\{\mu\in \M(T)\mid \mu(Z_{s,l})\ge 1/k\}.
\end{eqnarray*}

\textit{Claim.} $Z_{s,l}$ is  closed.

Let $(z_n)$ be a sequence in $Z_{s,l}$ such that $z_n\to z$. 
Since, for each $n\in\N$, $\overline{\gamma}(z_n,s)=\sup_{s\ge t}  \frac{\log \tau_{1/t}(z_n)}{\log t}\leq\alpha-1/l$, it follows that for each $t\ge s$, $\tau_{1/t}(z_n)\leq t^{\alpha-1/l}$. Now, fix $t\ge s$; then, there exists a sequence $(k_n)$, $k_n\in\N$, such that for each $n\in\N$, $k_n\leq t^{\alpha-1/l}$ and $d(T^{k_n}(z_n),z_n)\leq 1/t$.

Given that $(k_n)$ is bounded, there exist a sub-sequence $(k_{n_j})$,  $k\leq t^{\alpha-1/l}$ and  $j_0\in\N$ such that for each $j\ge j_0$, $k_{n_j}=k$. Since for each $j\in\N$, $d(T^{k_{n_j}}(z_{n_j}),z_{n_j})\leq 1/t$, it follows from the continuity of $T^k$ and the previous statements that $d(T^k(z),z)\leq 1/t$, which proves that $\tau_{1/t}(z)\leq t^{\alpha-1/l}$.  Given that $t\leq s$ is arbitrary, one gets $\sup_{t\ge s}  \frac{\log \tau_{1/t}(z)}{\log t}\leq\alpha-1/l$, which concludes that $z\in Z_{s,l}$.

\

Now, we show that $\M_{s,l}(k)$ is closed. Indeed, fix $s\in\N$ and  let $(\mu_n)$ be  a sequence in $\M_{s,l}(k)$ such that $\mu_n\to \mu$. Suppose that $\mu\notin \M_{s,l}(k)$; then, $\mu(A)>1-1/k$, where  $A=X\setminus Z_{s,l}$. Since, by Claim, $A$ is open in $X$ and  $\mu_n\to \mu$, it follows that $\liminf_{n\to \infty}\mu_n(A)\geq \mu(A)>1-1/k$, which shows that there exists an $\ell\in\N$ such that $\mu_{\ell}(A)>1-1/k$. This contradicts the fact that $\mu_{\ell}\in \M_{s,l}(k)$. Hence, $\mu\in \M_{s,l}(k)$.

Given that $\M_{s,l}(k)$ is closed, it follows that $\M(T)\setminus\overline{\mathcal{R}}(\alpha)=\bigcup_{s\in\N}\bigcup_{k\in \N}\bigcup_{l\in \N}\M_{s,l}(k)=\{\mu\in\M(T)\mid \mu\textrm{-}\essinf \overline{\gamma}_\mu(x,s)< \alpha\}$ is an $F_{\sigma}$ subset of $\M(T)$.
\end{proof}

The next results show that such sets are dense for the full-shift system.    

\begin{propo}
\label{denseretzer}
Let $(X,T)$ be the full-shift system, with $M$ a Polish metric space. Then, $\underline{\mathcal{R}}=\{ \mu\in \M_e \mid \mu\textrm{-}\esssup \underline{R}(x)=0\}$ is a dense subset of $\mathcal{M}_e$.
\end{propo}
\begin{proof}
  Note that if $\mu_x$ is a $T$-periodic measure, then for each $y\in\mathcal{O}(x)$, $R(y)=0$. The result follows, therefore, from the fact that the set of $T$-periodic measures is dense in $\M_e$ (this is Theorem~3.3 in~\cite{Parthasarathy1961}).
\end{proof}

\begin{propo}
\label{denseretpos}
Let $(X,T)$ be the full-shift system, with $M$ a perfect and compact metric space, and let $L>0$. Then, $\overline{\mathcal{R}}(L)=\{\mu\in \M_e\mid \mu\textrm{-}\essinf \overline{R}(x)\ge L\}$ is a dense subset of $\mathcal{M}_e$.
\end{propo}
\begin{proof}
  Fix $x\in X$,  $n\geq1$ and $\ve>0$. It follows from the argument presented in the proof of Lemma~\ref{dimpos01} that $B(x,\ve 2^{-n})\subset B(x,n,\ve)$ (recall that the full-shift system is Lipschitz continuous). Note that $\tau_{\ve 2^{-n}}(x)\geq R_n(x,\ve)$, where $R_n(x,\ve)=\inf\{k\geq 1\mid T^k(x)\in B(x,n,\ve)\}$ is the $n$th return time to the dynamical ball $B(x,n,\ve)$. Now, as in the proof of Proposition \ref{densepack}, for each $\mu\in\M(T)$ and each neighborhood $V_\mu(f_1,\ldots,f_n;\delta)$ (in the induced topology), there exists a measure $\zeta\in V_\mu(f_1,\ldots,f_n;\delta)\cap\M_e$ such that $h_{\zeta}(T)\geq L\log 2$. The result is now a consequence of Theorem A and Proposition A in \cite{Varandas}, which state that $\underline{R}(x)\geq \frac{h_{\zeta}(T)}{\log 2}=L$, for $\zeta$-a.e.$\,x$.
\end{proof}

\begin{rek}
\label{Rdenseretpos}
Proposition~\ref{denseretpos} can be extended to the case where $M$ is a Polish metric space using an adapted version of Lemma~\ref{dimpos01}. Namely, let $\mu\in\M_e$. It follows from the proof of Theorem~A in~\cite{Varandas} that $\underline{h}(T,x):=\lim_{\ve\to 0}\liminf_{n\to\infty}\frac{1}{n}\log R_n(x,\ve)$ is $T$-invariant (and therefore, constant for $\mu$-a.e.$\,x$; this constant may be infinite), where $R_n(x,\ve):=\inf\{k\in\N\mid f^k(x)\in B(x,n,\ve)\}$, and that $h_\mu(T)\le\underline{h}(T,x)$ for $\mu$-a.e.$\,x$; as in the proof of Katok's Theorem, this inequality is also valid for Polish spaces (see the discussion preceding Theorem~2.6 in~\cite{Riquelme}).
  
Since $R(x)$ is also $T$-invariant, it follows from the previous discussion and from the argument presented in the proof of Proposition A in \cite{Varandas} that, for $\mu$-a.e.$\;x\in X$,
\[\underline{R}(x)=\int \underline{R}(x)d\mu(x)\ge \int \frac{\underline{h}_{\mu}(T,x)}{\log 2}d\mu(x)\ge \frac{h_\mu(T)}{\log 2}.\] 
\end{rek}

\begin{proof3}
\begin{itemize}
\item[\textbf{V.}] The result is a consequence of Propositions \ref{Gdelta4},~\ref{denseretzer} and item~I of Theorem~\ref{teocentral3}. 
\item[\textbf{VI.}] It follows from Proposition~\ref{Gdelta4}, Remark~\ref{Rdenseretpos} and item~I of Theorem~\ref{teocentral3}, since $\overline{\mathcal{R}}=\bigcap_{L\ge 1}\overline{\mathcal{R}}(L)$.
\end{itemize}
\end{proof3}

\begin{rek} There is an alternative proof to the fact that $\underline{\mathcal{R}}=\{ \mu\in \M(T)\mid \mu\textrm{-}\esssup \underline{R}(x)=0\}$ is residual in $\mathcal{M}(T)$. In fact, this result can be seen as a direct consequence of Theorem 2 in \cite{Barreiral2001} and item~III of Theorem \ref{teocentral3}, since it follows that, for each $\mu\in HD$, $\mu\textrm{-}\esssup\underline{R}(x)\leq \dim_H(\mu)=0$.
\end{rek}

The following result states that each typical measure obtained in Theorem \ref{teocentral3} is supported on the dense $G_{\delta}$ set $\mathfrak{R}=\{x\in X\mid \underline{R}(x)=0 \mbox{ and } \overline{R}(x)=\infty\}$. 

\begin{propo}
\label{genericpoints}
Let $(X,T)$ the full-shift system, where $M$ is a perfect Polish metric space. Then, each of the sets $\mathfrak{R}^{-}=\{x\in X\mid \overline{R}(x)=\infty\}$ and $\mathfrak{R}_{-}=\{x\in X\mid  \underline{R}(x)=0\}$ is a dense $G_{\delta}$ subset of $X$. Moreover, for each $\mu\in\mathcal{\overline{R}}\cap\mathcal{\underline{R}}\cap C_X$, $\mu(\mathfrak{R}^{-}\cap\mathfrak{R}_{-})=1$.
\end{propo}
\begin{proof} We just present the proof that $\mathfrak{R}^{-}$ is a dense $G_{\delta}$ subset of $X$. 

  \emph{$\mathfrak{R}^{-}$ is a $G_{\delta}$ set in X.} Let  $\alpha>0$, $s\in \N$, and set $Z_{s}(\alpha)=\{x\in X\mid \overline{\gamma}(x,s)\le\alpha\}$. Following the proof of Claim in Proposition \ref{Gdelta4}, it is clear that $Z_{s}(\alpha)$ is  closed. Thus, taking $\alpha=n\in\N$, it follows that $\mathfrak{R}^{-}=\bigcap_{n\in \N}\bigcap_{s\in \N} (X\setminus Z_s(n))$ is a $G_{\delta}$ set in $X$.

\emph{$\mathfrak{R}^{-}$ is dense in $X$}. Let $\mu\in\mathcal{\overline{R}}\cap C_X$. Then, $\mu(\mathfrak{R}^{-})=1$. Suppose that $\mathfrak{R}^{-}$  is not dense; so, there exist $x\in X$ and $\ve>0$ such that $B(x,\ve)\cap\mathfrak{R}^{-}=\emptyset$. This implies that $1=\mu(\mathfrak{R}^{-}) + \mu(B(x,\ve))$, which is an  absurd, since $\mu(B(x,\ve))>0$.
\end{proof}

Now, we present equivalent results, to those already obtained in this section, for the quantitative waiting time indicators.

Let, for each $(x,y)\in X\times X$ and each $s\in\N$,
\[\overline{\gamma}(x,y,s):=\sup_{t>s}  \frac{\log \tau_{1/t}(x,y)}{\log t},\qquad\underline{\gamma}(x,y,s):=\inf_{t>s}  \frac{\log \tau_{1/t}(x,y)}{\log t}.\] 

\begin{propo}
\label{Gdelta5} 
Let $(X,T)$ be a topological dynamical system and let $\alpha>0$. Then, the set 
\begin{eqnarray*}
\underline{\mathscr{R}}(\alpha)&=&\{ \mu\in \M(T) \mid (\mu\times \mu)\textrm{-}\esssup \underline{R}(x,y)\leq \alpha\}
\end{eqnarray*}
is a $G_{\delta}$ subset of $\M(T)$.
\end{propo}
\begin{proof}
  Using the same ideas presented in the proof of the Proposition~\ref{Gdelta4}, we show that $\underline{\mathscr{R}}(\alpha)$ is a $G_\delta$ subset of $\M(T)$ by showing that $\M(T)\setminus \underline{\mathscr{R}}(\alpha)=\bigcup_{s\in\N}\{ \mu\in \M(T) \mid (\mu\times \mu)\textrm{-}\esssup \underline{\gamma}_\mu(x,y,s)>\alpha\}$ is an $F_{\sigma}$ set.

  Let 
  $l,s\in\N$, set $Z_{s,l}=\{(x,y)\in X\times X\mid \underline{\gamma}(x,y,s)\geq \alpha+1/l\}$ and set, for each $k\in\N$, 
\begin{eqnarray*}
\M_{s,l}(k)=\{\mu\in\M(T)\mid (\mu\times \mu)(Z_{s,l})\ge 1/k\}.
\end{eqnarray*}

The proofs that $Z_{s,l}$ and $\M_{s,l}(k)$ are closed sets follow the same arguments presented in the proof of Proposition~\ref{Gdelta4} (here, one uses Theorem~8.4.10 in~\cite{Bogachev} for the product measure $\mu\times\mu$).
Finally, since $\M(T)\setminus\underline{\mathscr{R}}(\alpha)=\bigcup_{s\in\N}\bigcup_{l\in\N}\bigcup_{k\in\N} \M_{s,l}(k)$, we are done. 
\end{proof}

\begin{proof3}
\begin{itemize}
\item[\textbf{VII.}] Since, by Proposition~\ref{Gdelta5}, $\underline{\mathscr{R}}=\{\mu\in \M_e\mid (\mu\times \mu)\textrm{-}\esssup\underline{R}(x,y)=0\}=\cap_{k\ge 1}\underline{\mathscr{R}}(1/k)$, one just needs to prove that $\underline{\mathscr{R}}$ is dense.  Let $\mu_z$ be a $T$-periodic measure. Then, for each $x,y\in\mathcal{O}(z)$, $\underline{R}(x,y)=0$. The result follows from the fact that the set of $T$-periodic measures is dense in $\M_e$ and from item~I in Theorem~\ref{teocentral3}.
%
%
  
\item[\textbf{VIII.}] The result is a direct consequence of Theorem~\ref{teocentral3} (IV) and the second inequality in~\eqref{Gala} (see Theorem~4 in \cite{Galatolo}). 
\end{itemize}
\end{proof3}

\

\begin{propo}
\label{genericpoints1}
Let $(X,T)$ the full-shift system, where $M$ is a perfect Polish metric space. Then, each of the sets $\mathfrak{S}^{-}=\{(x,y)\in X\times X\mid \overline{R}(x,y)=\infty\}$ and $\mathfrak{S}_{-}=\{(x,y)\in X\times X\mid  \underline{R}(x,y)=0\}$ is a dense $G_{\delta}$ subset of $X\times X$. Moreover, for each $\mu\in\overline{\mathscr{R}}\cap\underline{\mathscr{R}}\cap C_X$, $(\mu\times\mu)(\mathfrak{S}^{-}\cap\mathfrak{S}_{-})=1$.
\end{propo}

\section{Appendix}

\begin{proof4}
Since the arguments in both proofs (for Hausdorff and packing dimensions) are similar, we just prove the statement for  $\dim_{P}^+(\mu)$ and $\dim_{P}^-(\mu)$.

\textbf{a)} $\dim_{P}^+(\mu)=\mu\textrm{-}\esssup \overline{d}_{\mu}(x)$. Let $\alpha\ge 0$. We show that if  $\mu\textrm{-}\esssup\overline{d}_{\mu}(x)\leq \alpha$, then $\dim_{P}^+(\mu)\leq \alpha$. In fact, since $\mu\textrm{-}\esssup \overline{d}_{\mu}(x)=\inf\{a\in\R\mid \mu(\{x\mid \overline{d}_{\mu}(x)\leq a\})=1\}\leq \alpha$, one has $\mu(\{x\in X\mid \overline{d}_{\mu}(x)\leq \alpha\})=1$. It follows from the Definition~\ref{HPdim} that $\dim_{P}^+(\mu)\leq \dim_{P}(\{x\in X\mid \overline{d}_{\mu}(x)\leq \alpha\})$. Now, by Corollary 3.20(a) in \cite{Cutler1995}, one has $\dim_{P}(\{x\in X\mid \overline{d}_{\mu}(x)\leq \alpha\})\leq \alpha$. Thus, $\dim_{P}^+(\mu)\leq \alpha$.
   
Conversely, we show that if $\dim_{P}^+(\mu)\leq \alpha$, then $\mu\textrm{-}\esssup\overline{d}_{\mu}(x)\leq \alpha$. Suppose that there exists $\delta>0$ such that $\mu\textrm{-}\esssup\overline{d}_{\mu}(x)\ge\alpha+\delta$; then, by the definition of essential supremum of a measurable function, there exists $E\in \mathcal{B}$, with $\mu(E)>0$, such that for each $x\in E$, $\overline{d}_{\mu}(x)\geq \alpha+\delta/2$. Then, by Corollary 3.20(b) in \cite{Cutler1995}, $\dim_{P}(E)\ge\alpha+\delta/2$, and therefore, $\dim_P^+(\mu)\ge\alpha+\delta$/2. This contradiction shows that $\mu\textrm{-}\esssup\overline{d}_{\mu}(x)\leq \alpha$.

\textbf{b)} $\dim_P^-(\mu)=\mu\textrm{-}\essinf \overline{d}_{\mu}(x)$. Let $\alpha>0$. We show that if  $\mu\textrm{-}\essinf\overline{d}_{\mu}(x)\geq \alpha$, then $\dim_{P}^-(\mu)\geq \alpha$. By the definition of essential infimum of a measurable function, $\mu(A)=1$, where $A:=\{x\in X\mid \overline{d}_{\mu}(x)\geq \alpha\}$. Since, for each $E\in \mathcal{B}$, $\mu(E)=\mu(A\cap E)$ ($E\setminus A\subset A^c$), one may only consider, without loss of generality, those sets $E\in \mathcal{B}$ such that $E\subset A$. Thus, for each $A\supset E\in \mathcal{B}$ so that $\mu(E)>0$, it follows from Corollary 3.20(b) in \cite{Cutler1995} that $\dim_{P}(E)\geq \alpha$. The result is now a consequence of Definition~\ref{HPdim}.

Conversely, we show that if $\dim_{P}^-(\mu)\geq \alpha$, then $\mu\textrm{-}\essinf\overline{d}_{\mu}(x)\geq \alpha$. Suppose that there exists $\delta>0$ such that $\mu\textrm{-}\essinf\overline{d}_{\mu}(x)\leq \alpha-\delta$; then, by the definition of essential infimum of a measurable function, there exists $E\in \mathcal{B}$, with $\mu(E)>0$, such that for each $x\in E$, $\overline{d}_{\mu}(x)\leq \alpha-\delta/2$. Thus, $E\subset \{x\in X\mid \overline{d}_{\mu}(x)\leq \alpha-\delta/2\}=C$ and $\dim_P E\leq \dim_P C$. Then, by Corollary 3.20(a) in \cite{Cutler1995}, $\dim_{P}(E)\leq\alpha-\delta/2$, and therefore, $\dim_P^-(\mu)\leq\alpha-\delta/2$. This contradiction shows that $\mu\textrm{-}\esssup\overline{d}_{\mu}(x)\geq \alpha$.
\end{proof4}

\section*{Acknowledgments}
The first author was partially supported by FAPEMIG (a Brazilian government agency; Universal Project 001/17/CEX-APQ-00352-17). The  second author was partially
supported  by CIENCIACTIVA C.G. 176-2015.



\bibliography{refs}{}
\bibliographystyle{acm}

\end{document}